\documentclass[12pt]{amsart}
\usepackage{amscd}
\usepackage{amsmath}
\usepackage{verbatim}
\usepackage{enumitem}
\usepackage[cp1251]{inputenc}
\usepackage[all]{xy}
\usepackage{graphicx}
\usepackage{amssymb}

\numberwithin{equation}{section}
\newtheorem{thm}[equation]{Theorem}
\newtheorem{prop}[equation]{Proposition}
\newtheorem{problem}[equation]{Problem}
\newtheorem{lemma}[equation]{Lemma}
\newtheorem{cor}[equation]{Corollary}

\theoremstyle{definition}
\newtheorem{rem}[equation]{Remark}

\newcommand{\K}{\operatorname{\mathrm{K}}}
\newcommand{\EU}{\operatorname{\mathrm{EU}}}
\newcommand{\EO}{\operatorname{\mathrm{EO}}}
\newcommand{\GG}{\operatorname{\mathrm{G}}}
\newcommand{\GL}{\operatorname{\mathrm{GL}}}
\newcommand{\Sp}{\operatorname{\mathrm{Sp}}}
\newcommand{\SL}{\operatorname{\mathrm{SL}}}
\newcommand{\St}{\operatorname{\mathrm{St}}}
\newcommand{\StP}{\operatorname{\hat{\mathrm{P}}}}
\newcommand{\StL}{\operatorname{\hat{\mathrm{L}}}}
\newcommand{\StU}{\operatorname{\hat{\mathrm{U}}}}
\newcommand{\SO}{\operatorname{\mathrm{SO}}}
\newcommand{\E}{\operatorname{\mathrm{E}}}
\newcommand{\M}{\operatorname{\mathrm{M}}}
\newcommand{\sr}{\operatorname{\mathrm{sr}}}
\newcommand{\asr}{\operatorname{\mathrm{asr}}}
\newcommand{\diag}{\operatorname{\mathrm{diag}}}
\newcommand{\rA}{\mathrm{A}}
\newcommand{\rB}{\mathrm{B}}
\newcommand{\rC}{\mathrm{C}}
\newcommand{\rD}{\mathrm{D}}
\newcommand{\rE}{\mathrm{E}}

\def\ssub#1{\mathchoice
   {_{\lower2pt\hbox{$\scriptstyle #1$}}}
   {_{\lower2pt\hbox{$\scriptstyle #1$}}}
   {_{\lower1.5pt\hbox{$\scriptscriptstyle #1$}}}
   {_{\!\lower1.5pt\hbox{$\scriptscriptstyle #1$}}}}
\def\haruu#1{\ar@{-}[ur]^{\scriptstyle #1}}
\def\hardu#1{\ar@{-}[dr]^{\scriptstyle #1}}
\def\harud#1{\ar@{-}[ur]_{\scriptstyle #1}}
\def\hardd#1{\ar@{-}[dr]_{\scriptstyle #1}}
\def\lnode#1{*[o]{\circ}}
\def\lwnode#1{*[o]{\composite{!{\circ}*h!<0pt,-1.0em>{\scriptscriptstyle #1}}}}
\def\lbnode#1{*[o]{\composite{!{\bullet}*h!<0pt,-1.0em>{\scriptscriptstyle #1}}}}
\def\harr#1{\ar@{-}[r]_{\scriptscriptstyle #1}}
\def\harl#1{\ar@{-}[l]^{\scriptscriptstyle #1}}

\usepackage[hypertex]{hyperref}

\title{Improved stability for odd-dimensional orthogonal group.}


\author {Sergey Sinchuk}

\email {sinchukss {\it at} yandex.ru}

\date {\today}

\thanks {The author is grateful to N.~Vavilov, A.~Stepanov and A.~Luzgarev for their numerous comments which helped to improve the present text.
The work of the author has been supported by RFBR projects 11-01-00811, 13-01-91150, 13-01-92699, 
by the task projects 6.15.564.2013, 6.15.2174.2013 and by JSC "Gazprom Neft".}

\begin{document}

\begin{abstract}
We compute the kernel of the stabilization map for $\K_1$-functors modeled on split Chevalley groups of types $\rB_l, \rC_l, \rE_l$ one step below the stable range.
For the groups of type $\rB_l$ this implies early injective stability for $\K_1(\rB_l, R)$ over a certain class of rings.
\end{abstract}

\maketitle

%\tableofcontents

\section {Introduction}\label{intro}

Let $R$ be a commutative ring with a unit and $\Phi_l$ be a reduced irreducible root system of rank $l>1$.
Denote by $\K_1(\Phi_l, R) = \GG(\Phi_l, R) / \E(\Phi_l, R)$ the quotient of the simply connected Chevalley group of type $\Phi_l$ over $R$ by the elementary subgroup.
Following \cite{St}, we call this group {\it the value of unstable $\K_1$-functor modeled on Chevalley group of type $\Phi_l$ on a ring $R$}.
An embedding of root systems $\Psi \hookrightarrow \Phi$ induces a map between the corresponding groups $\theta^{\K_1}_{\Psi\hookrightarrow\Phi}\colon\K_1(\Psi, R) \rightarrow \K_1(\Phi, R).$

Finding conditions on $R$ sufficient for the injectivity or surjectivity of $\theta$ is a classical problem, which dates back to H.~Bass' paper~\cite{Bs}.
In the case $\rA_{l-1}\hookrightarrow \rA_l$ such conditions are stated in terms of the stable rank of $R$ (see \cite{BMS}, \cite{Vas0}).
More precisely, the map $\theta_{\rA_{l-1}\hookrightarrow\rA_l}^{\K_1}$ is surjective when $\sr(R)\leq l$ and is injective when $\sr(R)\leq l-1$.
The case of a general $\Phi$ has been exhaustively studied by M.~Stein in \cite{St}.

It is natural to attempt to find additional assumptions on $R$ which imply bijectivity of $\theta^{\K_1}$ below the stable range,
i.\,e. in the situation when it is not possible to apply the stability theorems of Bass, Vaserstein and Stein directly.
For example, in \cite{RK} it has been shown that the map $\theta^{\K_1}_{\rA_{l-1}\hookrightarrow\rA_l}$ is bijective, 
when $R$ is a nonsingular affine algebra of dimension $d$ over a perfect $C_1$-field and $l\geq d+1$.

On the other hand, one may attempt to describe the generators of the kernel of $\theta^{\K_1}_{\Psi\hookrightarrow\Phi}$ explicitly.
Let us state (the absolute case) of the main result of \cite{Ka}, the so-called ``presta\-bilization'' theorem.

Denote by $\widetilde{\E}(n, R)$ the normal closure (inside $\GL(n, R)$) of the subgroup spanned by the subgroup $\E(n, R)$, 
the mixed commutator subgroup $[\E(n, R), \GL(n, R)]$ (which is not contained in $\E(n, R)$ only for $n=2$),
and the generators of the form $(e_n + xy)(e_n + yx)^{-1},$ where $y=\diag(\xi, 1, \ldots, 1),\ \xi\in R$, $x\in\M(n,R)$ are such that $e_n+xy\in \GL(n, R).$

\begin{thm} \label{prestA} 
 Let $R$ be a commutative ring and $\max(\sr(R), 2)\leq n.$
 Then one has $$\GL(n, R) \cap \E(n+1, R) = \widetilde{\E}(n, R).$$
 In other words, the kernel of $\theta^{\K_1}_{\rA_{n-1}\hookrightarrow\rA_n}$ is generated by $\widetilde{\E}(n, R)/\E(n, R)$.
\end{thm}

The main purpose of the present article is to obtain an analogue of this theorem for split simple Chevalley groups of type $\rB_l, \rC_l, \rE_l$.

Fix a basis of simple roots $\Pi=\{\alpha_1,\ldots,\alpha_l\}$ for $\Phi_l$. The simple roots are numbered as in~\cite{Bou}.
Denote by $\Delta_i$ the subsystem of $\Phi_l$ spanned by all simple roots except the $i$-th one.

Set $i=l$, $j=1$ for $\Phi_l=\rB_l,\rC_l$ and $i=1$, $j=l$ for $\Phi_l=\rE_l$. We use this notation throughout the rest of the article.
Clearly, the subsystems $\Delta_i$ are of type $\rA_{l-1}$ in the case $\Phi_l=\rB_l, \rC_l$ and of type $\rD_{l-1}$ in the case $\Phi_l=\rE_l$.
On the other hand, the subsystems $\Delta_j$ have the same type as $\Phi_l$ (except for $\Phi_l=\rB_2, \rC_2, \rE_6$).

\begin{thm} \label{prestBC} 
Let $R$ be a commutative ring. Assume that one of the following holds:
\begin{enumerate}
 \item $\Phi_l = \rB_l, \rC_l$ and $\mathrm{max}(\sr(R), 2)\leq l-1;$
 \item $\Phi_l = \rE_l$, $l=6,7,8$ and $\asr(R)\leq l-2.$
\end{enumerate} 
Condiser the following diagram of unstable $\K_1$-groups induced by the natural embeddings of root systems.
$$ \xymatrix{ \mathrm{Ker}(\theta_1) \ar@{^{(}->}[r] \ar[d]_{\theta_2''} & \K_1(\Delta_i\cap\Delta_j, R) \ar[r]_{\theta_1} \ar[d]_{\theta_2} & \K_1(\Delta_i, R) \ar[d]_{\theta_2'} \\
              \mathrm{Ker}(\theta_1') \ar@{^{(}->}[r] & \K_1(\Delta_j, R) \ar[r]_{\theta_1'} & \K_1(\Phi_l, R)} $$
Then the map $\theta_2''$ is surjective. \end{thm}

Denote by $\theta$ the map $\GG(\rA_{l-1}, R)\hookrightarrow \GG(\Phi_l, R)$ induced by the embedding $\rA_{l-1}=\Delta_l\hookrightarrow \rD_l$
of root systems. From theorems~\ref{prestA}--\ref{prestBC} and the main result of \cite{RK} one can immediately deduce the following statement.
\begin{cor} \label{earlyB} \strut  \begin{enumerate} 
 \item Assume that $\max(2, \sr(R))\leq l$ and $\Phi_l=\rB_l, \rC_l$. Then  
 $$\GG(\rB_l, R)\cap\E(\rB_{l+1}, R) = \theta(\widetilde{\E}(l, R))\cdot\E(\rB_l, R).$$
 \item Assume that $R$ is a nonsingular algebra of dimension $d\geq 2$ over a perfect $C_1$-field. 
       Then the map $\K_1(\rB_{d+1}, R) \rightarrow \K_1(\rB_{d+2}, R)$ is an isomorphism
\end{enumerate}
\end{cor}
The definition of $C_1$-field can be found in~\cite[3.2~Ch.~II]{Se}.
In fact, the main result of \cite{RK} (and hence the above statement) holds under some weaker assumptions on $R$ (see~\cite[Prop.~3.1]{RK}).

Another consequence of the Theorem~\ref{prestBC} is the following result which establishes a connection between stability problems for the embeddings $\rD_5\hookrightarrow\rE_6$ and $\rD_6\hookrightarrow\rE_7$.
\begin{cor} \label{prestE}
Let $R$ be a commutative noetherian ring such that $\mathrm{dimMax}\leq 4$.
\begin{enumerate} \item There exists the following commutative diagram with exact rows.
$$ \xymatrix{ \mathrm{Ker}(\theta_1) \ar@{^{(}->}[r] \ar@{->>}[d]_{\theta_2''} & \K_1(\rD_5, R) \ar@{->>}[r]_{\theta_1} \ar[d]_{\theta_2} & \K_1(\rD_6, R) \ar[d]_{\theta_2'} \\
              \mathrm{Ker}(\theta_1') \ar@{^{(}->}[r] & \K_1(\rE_6, R) \ar@{->>}[r]_{\theta_1'} & \K_1(\rE_7, R)}$$
\item \label{pt2} The map $\mathrm{Ker}(\theta_2)\rightarrow \mathrm{Ker}(\theta_2')$ induced by $\theta_1$ is an epimorphism.
\item \label{pt3} The map $\K_1(\rE_6, R)/\mathrm{Im}(\theta_2)\cong \K_1(\rE_7, R)/\mathrm{Im}(\theta_2')$ induced by $\theta_1'$ is an isomoprhism of pointed sets.
\end{enumerate} \end{cor}
\begin{proof} Under the above assumption on $R$ the maps $\theta_1$, $\theta_1'$ are epimorphisms by \cite[Cor.~3.2]{St}, \cite[Th.~1]{Pl}.
Statements \ref{pt2}--\ref{pt3} follow from the nonabelian snake lemma. \end{proof}

The proof of our main results follows \cite{St} and is essentially based on the calculations with elementary root unipotents
and so-called ``stable'' calculations in the representations of Chevalley groups, i.\,e. the calculations with the highest weight vector.

\section {Principal notation}\label{not}

By commutator $[x,y]$ of elements $x$, $y$ we always mean the left-normed commutator $xyx^{-1}y^{-1}$.
We denote by $x^y$ the conjugate element $y^{-1}xy$.

For any collection of subsets $H_1,\ldots, H_n$ of a group $G$ we denote by $H_1\ldots H_n$ their Minkowski set-product,
i.\,e. the set consisting of arbitrary products $h_1\ldots h_n$ of elements $h_i\in H_i$. In particular, the equality
$G = H_1\cdot\ldots\cdot H_n$ means that every element $g\in G$ can be presented as a product $h_1\ldots h_n$ for $h_i\in H_i$.

We denote by $H\ltimes N$ the semidirect product of groups $H$ and $N$ such that $N$ is a normal subgroup in $H\ltimes N$.

\subsection[ssStein]{Chevalley groups and Steinberg groups}
Our treatment of Steinberg groups, Chevalley groups and their representations follows \cite{PSV}, \cite{V}, \cite{VP}.

Let $\Phi$ be a reduced irreducible system of rank $l$ and let $\Pi=\{\alpha_1,\ldots\alpha_{l}\}$ be some fixed basis of simple roots of $\Phi$.
Denote by $\Phi^+$ and $\Phi^-$ the subsets of positive and negative roots of $\Phi$ with respect to $\Pi$.
Denote by $m_r(\alpha)$ the $r$-th coefficient of the expansion of $\alpha$ in $\Pi$, i.\,e. $\alpha = \sum_r m_r(\alpha) \alpha_r$.
Obviously, the condition $m_r(\alpha)=0$ is equivalent to $\alpha\in\Delta_r$.

Denote by $\GG(\Phi, -)$ the simply connected Chevalley---Demazure group scheme with the root system $\Phi$, and by $\E(\Phi, -)$ its elementary subfunctor.
For $l\geq 2$ Taddei's normality theorem (see \cite[Th.~0.3]{Ta}) asserts that $\E(\Phi, R)\trianglelefteq\GG(\Phi, R)$.

The elementary subgroup is generated by the {\it elementary root unipotents} $t_\alpha(\xi)$ for all $\alpha\in\Phi$, $\xi\in R$.
These elements satisfy the Steinberg relations:
\begin{equation}\label{R1} t_\alpha(s) t_\alpha(t) = t_\alpha(s+t), \end{equation} 
\begin{equation}\label{R2} [t_\alpha(s),  t_\beta(t)] = \prod t_{p\alpha + q\beta}(N_{\alpha,\beta, p, q} s^p t^q),\quad \alpha\neq-\beta. \end{equation}
In the above formula $\alpha,\beta\in\Phi$, $s,t\in R$, and $p$, $q$ run over all positive integers such that $p\alpha + q\beta\in\Phi$.
The constants $N_{\alpha, \beta, p, q}$ are small integers which can be explicitly computed and only depend on $\Phi$.

The {\it Steinberg group} $\St(\Phi, R)$ is defined by the generators $x_{\alpha}(\xi)$, $\alpha\in\Phi$, $\xi\in R$ and relations
\ref{R1}, \ref{R2} with the $t_\alpha(\xi)$'s replaced by $x_\alpha(\xi)$'s. Basic properties of these groups are discussed in \cite{St2}.

We denote by $X_\alpha$ the root subgroup corresponding to the root $\alpha$, i.\,e. the subgroup consisting of the root unipotents $x_{\alpha}(\xi)$, $\xi\in R$.

An embedding of root systems $\Psi \subseteq \Phi$ induces the natural transformations 
$$\theta_{\Psi\hookrightarrow\Phi}\colon\GG(\Psi, -)\rightarrow \GG(\Phi, -),\ \theta^{\St}_{\Psi\hookrightarrow\Phi}\colon\St(\Psi, -) \rightarrow \St(\Phi, -).$$
Notice that the maps $\theta^{\St}_{\Psi\hookrightarrow\Phi}(R)$ are not injective in general.

For $1\leq r\leq l$ and $s\neq r$ consider the following subgroups of $\St(\Phi, R)$:
$$\begin{array}{cclccc}
 \StP_r  (\Phi_l, R) & = & \{x_\alpha(\xi), m_r(\alpha)\geq 0\}, & \StL_r(\Phi_l, R) & = & \{x_\alpha(\xi), m_r(\alpha) = 0\}, \\
 \StU_r  (\Phi_l, R) & = & \{x_\alpha(\xi), m_r(\alpha) > 0\}, & \StU^-_r(\Phi_l, R) & = & \{x_\alpha(\xi), m_r(\alpha) < 0\}, \\
 \StU^-_{rs} (\Phi_l, R)& = & \StU^-_r \cap \StU^-_s.
\end{array}$$
Clearly $\StL_r(\Phi_l, R)$ normalizes both $\StU_r(\Phi_l, R)$ and $\StU^-_r(\Phi_l, R)$, 
hence $\StP_r(\Phi_l, R)$ admits the Levi decomposition: $$\StP_r(\Phi_l, R) = \StL_r(\Phi_l, R) \ltimes \StU_r (\Phi_l, R).$$
Denote by $\StU(\Phi, R)$ (respectively, by $\StU^-(\Phi, R)$) the subgroup spanned by $x_\alpha(\xi)$ for $\xi\in R$ and $\alpha\in\Phi^+$ (respectively, $\alpha\in\Phi^-$).
If the choice of $\Phi_l$ is clear from the context, we shorten the notation for $Z(\Phi_l, R)$ to $Z$, where 
$$Z=\GG, \E, \StU, \StU^-, \StP_r, \StL_r, \StU_r, \StU_r^-, \StU_{rs}^-.$$

\subsection[ssChev]{Representations of Chevalley groups.}
In the present article we work with the irreducible fundamental representations of $\GG(\Phi_l, R)$ corresponding to the highest weight $\varpi_j$,
i.\,e. $\varpi_1$ in the case of a classical $\Phi_l$ and $\varpi_l$ in the case $\Phi_l=\rE_l$. 
The former are the natural vector representations of the classical groups acting on the free modules $V=R^n$ of dimension $n=2l,2l+1,2l,2l$ for $\Phi=\rA_l,\rB_l,\rC_l,\rD_l$ respectively.
The latter act on the free modules $V=R^n$ of dimension $n=27, 56, 248$ for $l=6,7,8$ respectively.
All these representations are {\it basic} in the sense of \cite{PSV}.
Moreover, all the representations except for $(\rB_l, \varpi_1)$ and $(\rE_8, \varpi_8)$ are {\it microweight}, i.\,e. they do not have zero weights.

For classical groups we use the standard numbering of the weights of represen\-tations (cp.~\cite[§1B]{St}).
In particular, the coordinates of the elements of $V$ are indexed as follows:
$$\begin{array}{cll}
  1,2,\ldots, l+1 & \text{in the case} & \Phi =\rA_l, \\
  1,2,\ldots l, 0, -l,\ldots, -2, -1 & \text{in the case} & \Phi =\rB_l, \\
  1,2,\ldots l, -l,\ldots, -2, -1 & \text{in the cases}   & \Phi =\rC_l, \rD_l. \\
\end{array}$$

\begin{figure}\nonumber
\begin{tabular}{c}
$\xymatrix{\lwnode{1}\harr{1}&\lwnode{2}\harr{2}&\ldots&\lbnode{l}\harl{l-1}\harr{l}&\lwnode{0}\harr{l}&\lwnode{-l}\harr{l-1}&\ldots\harr{2}&\lwnode{-2}\harr{1}&\lwnode{-1}} $\\ \\
$\xymatrix{\lwnode{1}\harr{1}&\lwnode{2}\harr{2}&\ldots&\lbnode{l}\harl{l-1}\harr{l}&\lwnode{-l}\harr{l-1}&\ldots\harr{2}&\lwnode{-2}\harr{1}&\lwnode{-1}}$ \\ \\
$\xymatrix @R=.1in @C=.1in {\lwnode{1}\hardd{6}&&&&&&&&\lbnode{-1}\hardu{1} \\
&\lwnode{2}\hardd{5}&&&&&&\lbnode{-2}\haruu{}\hardu{}&&\lnode{}\hardu{3} \\
&&\lwnode{3}\hardd{4}&&\lwnode{5}\hardd{} && \lbnode{-3}\haruu{}\hardu{} && \lnode{}\haruu{}\hardu{}&&\lnode{}\hardu{4}&&\lnode{}\hardu{2} \\
&&&\lwnode{4}\hardd{3}\haruu{}&&\lbnode{-4}\haruu{}\hardu{}&&\lnode{}\haruu{}\hardu{}&&\lnode{}\haruu{}\hardu{}&&\lnode{}\haruu{}\hardu{}&&\lnode{}\hardu{4} \\
&&&&\lbnode{-5}\hardd{1}\haruu{}&&\lnode{}\harud{4}&&\lnode{}\harud{5}&&\lnode{}\haruu{}\hardu{}&&\lnode{}\harud{5}&&\lnode{}\hardu{3} \\
&&&&&\lnode{}\harud{2}&&&&&&\lnode{}\harud{6}&&&&\lnode{}\hardu{1} \\
&&&&&&&&&&&&&&&&\lnode{}}$
\end{tabular}
\caption{The weight diagrams of $(\rB_l,\varpi_1)$, $(\rC_l, \varpi_1)$, $(\rE_6,\varpi_6)$.}\label{e6_weights}\end{figure}

As for exceptional groups, for our purposes it suffices to refer explicitly the weights of the representations $(\rE_l, \varpi_l)$,
arising from subrepresentation $(\rD_{l-1}, \varpi_1)$ corresponding to the same highest weight $\varpi_l$.
For such weights we keep the natural numbering introduced above. 
Our partial num\-bering of the weights in the case $(\rE_6, \varpi_6)$ is illustrated in Figure~\ref{e6_weights}.

We denote by $v^+$ the {\it highest weight vector} of a representation.
Clearly, in the above notation the coordinate $(v^+)_1$ is equal to $1$ and all other coordinates are zero.

The following easy observation directly follows from the description of the action of elementary root unipotents on $V$ (see~\cite[Lemma~2.3]{Ma}, \cite[\S~4.3]{V}).
\begin{rem}\label{rootUnipAction} Let $\Phi_l=\rE_l$, $l=6,7,8$ and assume that $v\in V$ is such that all coordinates, except maybe $v_{1},\ldots, v_{l-1}$ are zero. Then $x_{-\alpha\ssub{1}}(\xi)$ fixes $v$ for any $\xi\in R$. \end{rem}

\subsection[ssStab]{Stability conditions.}
Recall that a column $(a_1,\ldots,a_n)^t\in R^{n}$ is called {\it unimodular} if $R a_1+\ldots+ R a_n=R.$
A unimodular column $(a_1,\ldots,a_{n+1})^t\in R^{n+1}$ is called {\it stable} if there exist 
$b_1,\ldots b_n\in R$ such, that the column $(a_1 + b_1a_{n+1}, a_2 + b_2a_{n+1},\ldots, a_n + b_na_{n+1})^t$ is unimodular.

By definition, the {\it stable rank} of $R$ is the smallest natural number $k$ for which any unimodular column of height $>k$ is stable.
If such $k$ does not exist we assume the stable rank of $R$ to be equal to $\infty$.

For a column $a=(a_1,\ldots, a_n)^t\in R^n$ denote by $\mathfrak{L}(a)$ the intersection of the left maximal ideals of $R$ containing $a_1, \ldots, a_n$.
Clearly, $u\in R^n$ is unimodular if and only if $\mathfrak{L}(u)=R$.

By definition, the {\it absolute stable rank} is the smallest natural $k$ such that for any column $a=(a_1,\ldots,a_{n+1})^t$ of height $n+1>k$ there exist $b_1,\ldots b_n$ such that
$$\mathfrak{L}(a_1+b_1a_{n+1},\ldots,a_n+b_na_{n+1})^t= \mathfrak{L}a.$$

The stable rank and the absolute stable rank of $R$ are denoted by $\sr(R)$ and $\asr(R)$ respectively. Clearly, $\sr(R)\leq\asr(R)$. 
If $R$ is commutative and its maximal spectrum is Noetherian of dimension $d$ (i.\,e. $\mathrm{dimMax}(R)=d$) then by \cite[Th.~2.3]{EO} one has $\asr(R)\leq d+1$.

The following lemma is a direct consequence of the definition of stable rank.
\begin{lemma} \label{simpleLemma} Let $\sr(R)\leq l-1$, then for any unimodular column $v\in R^l=V$ there exist
matrices $x=\left(\begin{smallmatrix} e_{l-1}& * \\ 0 & 1 \end{smallmatrix}\right)$, $y=\left(\begin{smallmatrix} e_{l-1}& 0 \\ *& 1 \end{smallmatrix}\right)$
such that the $l$-th cooordinate of the vector $yx\cdot v$ is zero.\end{lemma}

Denote by $p$ the matrix which has $1$'s along its secondary diagonal and zeroes elsewhere. Call a matrix $a\in\M(l, R)$ {\it antipersymmetric} if $pa^tp = a$.
\begin{lemma}\label{LambdaS_lemma} Assume that $\asr(R)\leq l-1$, then for any columns $u^+, u^-\in R^l$ such that $(u^+, u^-)^t\in R^{2l}$ is unimodular there exists
an antipersymmetric matrix $a\in\M(l, R)$ such that $u^+ + au^- \in R^l$ is unimodular. \end{lemma}
\begin{proof} The statement of the lemma is a special case of \cite[Th.~1.2]{BPT} applied to $R$ viewed as a form ring with trivial involution and zero form parameter. \end{proof}

Denote by $H$ the hyperbolic embedding $H:\GL(l, R)\rightarrow\mathrm{O}(2l, R)$ of a general linear group into the even-dimensional orthogonal group,
i.\,e. the map which sends $g$ to $\left(\begin{smallmatrix} g & 0 \\ 0 & p(g^{t})^{-1}p \end{smallmatrix}\right)$.
\begin{lemma}\label{asrLemma} Assume that $\asr(R)\leq l-1$. Then for any unimodular column $v\in R^{2l}=V$ there exist orthogonal matrices 
$x=\left(\begin{smallmatrix} e_l & * \\ 0 & e_l \end{smallmatrix}\right)$, $y=\left(\begin{smallmatrix} e_l & 0 \\ * & e_l \end{smallmatrix}\right)$
such that $(yx\cdot v)_k=0$ for $k=-l,\ldots, -1.$ \end{lemma}
\begin{proof}
 Applying Lemma~\ref{LambdaS_lemma} we obtain $x=\left(\begin{smallmatrix} e_l & a \\ 0 & e_l \end{smallmatrix}\right)$ such
 that the first $l$ coor\-dinates of $x\cdot v$ form a unimodular column of height $l$.
 
 The group $\E(l, R)$ acts transitively on the set of unimodular columns of height $l$, therefore
 we can choose an appropriate $g\in\E(l, R)$ such that $$H(g)x\cdot v= (1,0,\ldots, 0,*,\ldots *)^t,$$
 where the number of stars equals $l$.
 Choose $y'=\left(\begin{smallmatrix} e_l & 0 \\ * & e_l \end{smallmatrix}\right)$ such that $y'H(g)x\cdot v = e_1 = (1,0,\ldots, 0)^t$.
 Clearly $y=y'^{H(g)}$ is an orthogonal matrix of the required form and $$yx\cdot v = y'^{H(g)} x\cdot v = H(g)^{-1} e_1 = (*,\ldots, *,0,\ldots, 0)^t.$$
 \end{proof}

\section {Proof of the main results.} \label {mainPart}

The main ingredient needed in the proof of the Theorem~\ref{prestBC} is the following factorization of the Steinberg group.
\begin{prop}\label{DVfactor} Let $\Phi_l$, $R$, $i$, $j$ be as in the statement of \ref{prestBC}. Then one has the following factorization
\begin{equation}\label{DVdecomp} \St(\Phi_l, R) =  \StU \cdot \StU^- \cdot \StL_i \cdot \StP_j.\end{equation} \end{prop}
From the Levi decomposition it follows that the decomposition~\ref{DVdecomp} can be rewritten in the form $\St(\Phi_l, R) = \StP_i\cdot \StU^-_{ij}\cdot \StP_j,$ 
i.\,e. the Proposition~\ref{DVfactor} can be thought of as an analogue of the ''Dennis---Vaserstein decomposition`` in the sense of~\cite{ST} (cf.~Lemma~2.1).

\begin{rem}
 The cases $\Phi_l=\rB_l,\rC_l$ of the factorization \ref{DVfactor} directly follow from \cite[Th.~2.5]{St}.
 Nevertheless, we reprove these cases below for the sake of self-containedness.
\end{rem}
The group $\StL_i$ coincides with the image of the stabilization map $\theta_{\Delta_{i}\hookrightarrow\Phi_l}$.
Denote by $\widetilde{S}$ the subset of $\StL_i$ consisting of elements $g\in\StL_i$ such that
\begin{itemize}
\item the coordinate $(\varphi(g)\cdot v^+)_l$ is zero in the cases $\Phi_l = \rB_l,\rC_l$;
\item the coordinates $(\varphi(g)\cdot v^+)_k$ are zero for $k=-1,\ldots,-(l-1)$ in the case $\Phi_l = \rE_l$.
\end{itemize}
In Figure~\ref{e6_weights} these coordinates are marked with black circles.

\begin{lemma}\label{part1}
For every $a\in \StL_i$ there exist $x\in \StL_i\cap \StU,$ $y\in\StL_i\cap \StU^-$ such that $yxa \in \widetilde{S}.$
\end{lemma}
\begin{proof}
 The restriction of the projection $\varphi$ to $\StU(\Phi_l, R)$ or $\StU^-(\Phi_l, R)$ is an isomoprhism, hence $x$ and $y$ are determined by the matrices $\varphi(x)$, $\varphi(y)$.
 It remains to apply Lemma~\ref{simpleLemma} in the cases $\Phi_l=\rB_l, \rC_l$ and Lemma~\ref{asrLemma} in the case $\Phi_l=\rE_l$ taking $v=\varphi(a)\cdot v^+$.
\end{proof}

\begin{lemma} \label{auxFacts} 
\begin{enumerate}[label=(\alph*)]
 \item\label{part2} The following inclusions hold
 \begin{equation}\nonumber
 \begin{array}{rcl}
 X_{-\alpha\ssub{i}}\cdot \StU &\subseteq & \StU \cdot X_{-\alpha\ssub{i}} \cdot X_{\alpha\ssub{i}}, \\
 X_{\alpha\ssub{i}}\cdot \StU^- &\subseteq & \StU^- \cdot X_{\alpha\ssub{i}} \cdot X_{-\alpha\ssub{i}}, \\
 X_{-\alpha\ssub{i}}\cdot \StU\cdot \StU^- &\subseteq & \StU \cdot \StU^- \cdot X_{\alpha\ssub{i}}\cdot X_{-\alpha\ssub{i}}. 
 \end{array}
 \end{equation}
 \item\label{part3} For any element $a$ of $\widetilde{S}$ one has $(X_{-\alpha\ssub{i}})^a\subseteq\StL_j$.
\end{enumerate} \end{lemma}
\begin{proof}
Denote by $\StU'$ the subgroup generated by the elements $x_\alpha(\xi)$ for $\alpha\in\Phi^+\setminus\{\alpha_i\}$, $\xi\in R$.
Let $u$ be an arbitary element of $\StU$. Express $u$ as a product $v x_{\alpha\ssub{i}}(\zeta)$ for 
some $v\in \StU'$, $\zeta\in R$. From the Chevalley commutator formula~\ref{R2} it follows that ${v}^{x_{-\alpha\ssub{i}}(\eta)}\in\StU$.
Thus $$x_{-\alpha\ssub{i}}(\eta) u = {v}^{x_{-\alpha\ssub{i}}(-\eta)}\cdot x_{-\alpha\ssub{i}}(\eta)\cdot x_{\alpha\ssub{i}}(\zeta) 
\in \StU \cdot X_{-\alpha\ssub{i}}\cdot X_{\alpha\ssub{i}}.$$
The second inclusion can be demonstrated in a similar way. The third inclusion follows from the first two.

From the Levi decomposition it follows that $z=x_{-\alpha\ssub{i}}(\xi)^a$ lies in $\StU^-_i$.
On the other hand, the element $t_{-\alpha\ssub{i}}(\xi)$ fixes $a\cdot v^+$ (cf.~Remark~\ref{rootUnipAction}), hence $v^+$ is fixed by $\varphi(z)$.
From the description of the action of elementary root unipotents on $V$ (see \cite[Lemma~2.3]{Ma}, \cite[\S~4.3]{V}) we conclude that $z\in \StU^-_i\cap\StL_j$ as claimed.
\end{proof}

\begin{proof}[Proof of the Proposition~\ref{DVfactor}.]  

Call a decomposition $g = u\cdot v\cdot a\cdot p$ of the form~\ref{DVdecomp} {\it reduced} if $a\in \widetilde{S}$.
Let us check that our stability assumptions imply that every such $g$ can be rewritten as a reduced decomposition.
Indeed, take $g=u\cdot v\cdot a\cdot p\in \StU \cdot \StU^- \cdot \StL_i \cdot \StP_j$.
Without loss of generality we may assume that $u\in \StU_i$, $v\in \StU^-_i$.
It remains to choose $x, y$ from the statement of Lemma~\ref{part1} and rewrite 
$$u\cdot v\cdot a\cdot p = u x^{-1} \cdot (x v x^{-1}) y^{-1} \cdot (yxa)\cdot p \in \StU \cdot \StU^- \cdot \widetilde{S} \cdot \StP_j.$$

Since the group $\St(\Phi_l, R)$ is spanned by the root subgroups $X_{\pm\alpha_k}$, $1\leq k\leq l$ it suffices to show
that the decomposition~\ref{DVdecomp} is stable under left multiplication by $X_{-\alpha\ssub{k}}$.
For $k\neq i$ this easily follows from the Levi decomposition.

Now take $k=i$ and let $a$ be an arbitary element of $\widetilde{S}$.
In view of Lemma~\ref{auxFacts} we have
$$ X_{-\alpha\ssub{i}}\cdot \StU\cdot \StU^-\cdot a\cdot \StP_j \subseteq 
\StU\cdot \StU^-\cdot X_{\alpha\ssub{i}}\cdot a\cdot (X_{-\alpha\ssub{i}})^a\cdot \StP_j \subseteq
\StU\cdot \StU^-\cdot a\cdot \StP_j.$$ \end{proof} 

Set $G'= \GG(\Delta_j, R)$. We identify $G'$ with the corresponding subgroup of $\GG(\Phi_l, R)$.
The proof of the following technical statement is similar to~\cite[Th.~3.1]{St}.
\begin{lemma}\label{DV_reduction} Assume that $\St(\Phi_l, R)$ admits decomposition~\ref{DVdecomp}.
Then one has \begin{equation}\label{DV_red} \St(\Phi_l, R)\cap \varphi^{-1}(G') \subseteq (\StL_i(\Phi_l, R)\cap \varphi^{-1}(G')) \cdot \StL_j(\Phi, R).\end{equation} \end{lemma}
\begin{proof}
Denote by $X$ the left hand side of the inclusion~\ref{DV_red}. Applying Proposition~\ref{DVfactor} we get that
\begin{equation}\nonumber X=(\StU_i \cdot \StU^-_i\cdot \StL_i\cdot \StP_j)\cap \varphi^{-1}(G') = (\StU_i\cdot\StL_i\cdot\StU^-_i\cdot\StU_{ij}\cdot\StL_j)\cap \varphi^{-1}(G').\end{equation}
Let $g$ be an arbitary element of $X$, write its decomposition $g = u_1\cdot l_1\cdot v_2\cdot u_2\cdot l_2$ for some $u_1$, $l_1$, $v_2$, $u_2$, $l_2$ belonging to the corresponding subgroups.

Since the elements $\varphi(g)$, $\varphi(u_1)$, $\varphi(u_2)$, $\varphi(l_2)$ fix $v^+$, we get that $\varphi(v_2)\cdot v^+ = \varphi(l_1^{-1})\cdot v^+.$
Comparing the coordinates of these vectors we conlcude that $\varphi(v_2)\cdot v^+ = v^+$, therefore $v_2\in\StU^-_i\cap \StL_j$ and
from the Levi decomposition it follows that $X \subseteq (\StL_i \cdot \StU_i \cdot \StL_j)\cap\varphi^{-1}(G').$

Denote by $\sigma$ the automorphism of $\St(\Phi_l, R)$ induced by the automorphism $\alpha\mapsto-\alpha$ of $\Phi_l$. %(см. \cite[\S~10]{Stn}).
Clearly $\sigma(X)\subseteq (\StL_i\cdot\StU^-_i\cdot\StL_j)\cap\varphi^{-1}(\GG(\Phi_l, R)).$

Consider an element  $g\in\sigma(X)$. As in the previous case, decompose $g=l_1' \cdot v_1'\cdot l_2'$ and notice that
$\varphi(v_1')\cdot v^+ = \varphi({l_1'}^{-1})\cdot v^+$, hence $v_1'\in\StL_i$ and
$X \subseteq (\StL_i\cdot\StL_j)\cap \varphi^{-1}(G') \subseteq  (\StL_i\cap \varphi^{-1}(G')) \cdot \StL_j$ as claimed.
\end{proof}

\begin{proof}[Proof of the Theorem~\ref{prestBC}.]
Applying $\varphi$ to both sides of~\ref{DV_red} we get the inclusion
\begin{multline}\nonumber
\E(\Phi_l, R)\cap \GG(\Delta_j, R) \subseteq (\E(\Delta_i, R)\cap \GG(\Delta_j, R))\cdot \E(\Delta_j, R) = \\ =
\theta_{\Delta_i\hookrightarrow\Phi\ssub{l}}(\E(\Delta_i, R)\cap\GG(\Delta_i\cap\Delta_j, R))\cdot \E(\Delta_j, R) \end{multline}
from which the theorem follows. \end{proof}

\section{Concluding remarks}\label{remarks}

We note that no explicit description of the stabilization kernel similar to Theorem~\ref{prestA} is known for $\Phi=\rD_l$.
Moreover, no result similar to Corollary~\ref{earlyB} may hold for $\Phi=\rD_l$ in general as the following negative result has been demonstrated in \cite{RBJ}.
\begin{thm} Let $R$ be a nonsingular affine algebra of dimension $d\geq 2$ over a perfect $C_1$-field. Asssume that $2\in R^*$.
\begin{enumerate} \item The bijectivity of the map $$\SO(2(d+1), R)/ \EO(2(d+1), R) \rightarrow \SO(2(d+2), R) / \EO(2(d+2), R)$$ implies the transitivity
of the action of $\E(d+1, R)$ on the unimodular columns of height $d+1$. \item There exists an algebra from the considered class of rings for which the latter condition fails.
\end{enumerate} \end{thm}

On the other hand, in fact, a much stronger version of the 2nd statement of Corollary~\ref{earlyB} holds for~$\Phi_l=\rC_l$ (see \cite[Th.~2]{BR1}).

We conclude the text with several problems related to the main result of the paper.
\begin{problem} Generalize Theorem~\ref{prestA} to the case of the quadratic group from \cite{BPT} using the subgroup $\widetilde{\EU}$ defined in \cite{Vas1} as an analogue of $\widetilde{\E}(n, R)$. \end{problem}
\begin{problem} Obtain a relative version of Theorem~\ref{prestBC}. \end{problem}
\begin{problem} In the assumptions of Theorem~\ref{prestBC} describe the kernel of $\theta_2''$. \end{problem}

\end{document}